\documentclass[11pt]{amsart}

\usepackage{amssymb}
\usepackage{amsthm}
\usepackage{amsmath}
\usepackage{mathtools}
\usepackage[all,knot]{xy}

\usepackage[utf8]{inputenc}
\usepackage[backend=biber, style=alphabetic, sorting=nyt]{biblatex}
\usepackage{hyperref}
\usepackage{float}
\usepackage{subcaption}
\usepackage{graphicx}

\newtheorem{thm}{Theorem}[section]
\newtheorem{lem}[thm]{Lemma}
\newtheorem{prop}[thm]{Proposition}
\newtheorem{cor}[thm]{Corollary}
\newtheorem{conj}[thm]{Conjecture}
\newtheorem{defn}[thm]{Definition}
\newtheorem{question}[thm]{Question}

\theoremstyle{remark}
\newtheorem*{rmk}{Remark}

\newcommand{\Z}{\mathbb Z}
\newcommand{\R}{\mathbb R}

\newcommand{\wt}{\widetilde}
\DeclareMathOperator{\br}{br}
\DeclareMathOperator{\lk}{lk}
\DeclareMathOperator{\con}{Con}

\usepackage{fullpage}

\begin{document}

\title{Do Link Polynomials Detect Causality In Globally Hyperbolic Spacetimes?}
\author{Samantha Allen}
\author{Jacob H. Swenberg}
\address{Samantha Allen: Department of Mathematics, Dartmouth College, NH 03755}\email{Samantha.G.Allen@dartmouth.edu}
\address{Jacob H. Swenberg: Department of Mathematics, Dartmouth College, NH 03755}\email{Jacob.H.Swenberg.21@dartmouth.edu}

\dedicatory{Dedicated to the memory of John Conway and Vaughan Jones.}

\maketitle

\begin{abstract}
Let $X$ be a $(2+1)$-dimensional globally hyperbolic spacetime with a Cauchy surface $\Sigma$ whose universal cover is homeomorphic to $\R^2$.
We provide empirical evidence suggesting that the Jones polynomial detects causality in $X$. 
We introduce a new invariant of certain tangles related to the Conway polynomial, and prove that the Conway polynomial does not detect the connected sum of two Hopf links among relevant 3-component links, which suggests that the Conway polynomial does not detect causality in the scenario described.
\end{abstract}

\section{Introduction}

%

Let $X$ be a $(2+1)$-dimensional globally hyperbolic spacetime with Cauchy surface $\Sigma$ homeomorphic to $\R^2$, and let $N$ be the set of future-directed null geodesics in $X$. The set $N$ can be identified with the spherical cotangent bundle $ST^*\Sigma$ of $\Sigma$, which in this case is homeomorphic to a solid torus $S^1\times \R^2$. The \emph{sky} of $x \in X$, denoted $S_x \subset N$, is the set of all future-directed null geodesics through $x$. The sky $S_x$ is homeomorphic to a circle, and viewed as a subset of the solid torus $S^1\times \R^2$, $S_x$ is isotopic to $S^1\times\{0\}$. For more explanation, see \cite{chernov_khovanov_2020} or \cite{natario_linking_2004}.

Nemirovski and Chernov \cite[Thm B]{chernov_legendrian_2010} proved the Low conjecture, which says that as long as $\Sigma$ is not a closed 2-manifold, two events $x,y \in X$ are causally related if and only if their skies $S_x\sqcup S_y$ are linked. In our context ($\Sigma$ homeomorphic to $\R^2$), \emph{linked} means either $S_x\cap S_y \neq \varnothing$, or $S_x\sqcup S_y$ is not isotopic in $N$ to $S^1\times\{a\}\sqcup S^1\times\{b\}$ for $a,b \in \R^2$. Nemirovski and Chernov actually proved more; they showed that the relationship between linking and causality holds as long as $\Sigma$ is not homeomorphic to $S^2$ or $\R P^2$. They also proved the Legendrian Low conjecture, which is analogous to the Low conjecture but for higher-dimensional spacetimes, where topological linking is replaced by Legendrian linking \cite[Thm A]{chernov_legendrian_2010}.

A natural question is whether linking of $S_x$ and $S_y$, and thus causality, could be detected by various link invariants. Nat\'ario and Tod \cite{natario_linking_2004} provided a large family of pairs of skies corresponding to causally related events such that, for each pair, the associated link has nontrivial Kauffman polynomial.  The Kauffman polynomial is related to the Jones polynomial $V(L)$ by a change of variables, but there is another invariant that contains strictly more information than both. Specifically, Khovanov homology provides a ``categorification" of the Jones polynomial \cite{bar-natan_khovanovs_2002}. In fact, Khovanov homology detects causality in $X$ \cite[Theorem 1]{chernov_khovanov_2020}. Another common link polynomial is the \emph{Alexander-Conway polynomial} (also called the \emph{Conway polynomial}) $\nabla(L)$, which is categorified by link Floer homology \cite{kauffman_alexanderconway_2016}. Chernov, Martin, and Petkova have suggested that link Floer homology will also detect causality in this setting \cite{chernov_khovanov_2020}. The related knot polynomials, the Jones polynomial and the Conway polynomial, are strictly weaker link invariants than their respective categorifications. On the other hand, the work of \cite{natario_linking_2004} indicates that these polynomials still might detect causality.

By a remark in \cite{chernov_khovanov_2020}, if an invariant can detect causality in $X$ as described, then causality can be detected in any $(2+1)$-dimensional globally hyperbolic spacetime $X'$ whose Cauchy surface $\Sigma'$ is not homeomorphic to $S^2$ or $\R P^2$. The reason is that the universal cover $\wt{X'}$ of such a spacetime is also a globally hyperbolic spacetime with Cauchy surface $\wt{\Sigma'}$ that is a universal cover of $\Sigma'$. By lifting of paths, causal relationships between points in $X'$ are equivalent to causal relationships between points in $\wt{X'}$. If $\Sigma'$ is not homeomorphic to $S^2$ or $\R P^2$, then the universal cover $\wt{\Sigma'}$ is homeomorphic to $\R^2$. As a result, it will be sufficient for our purposes to consider the case described above, where $\Sigma$ is homeomorphic to $\R^2$.

In this paper, we show that the Conway polynomial does not detect causality (linking) in the given setting, and we conjecture that the Jones polynomial does detect causality. At first glance, this hypothesis might seem strange: there are infinitely many nontrivial $k$-component links indistinguishable from unlinks via their Jones polynomials \parencite[Cor. 3.3.1]{eliahou_infinite_2003}. There are links with similarly trivial Conway polynomials, such as L10n32 and L10n59 \cite{linkinfo}.
However, there are topological restrictions on the links being considered. Components of links given by skies of events exist in a solid torus, and each sky must be isotopic to a longitude of the solid torus $S^1\times\R^2$. While this solid torus has a natural embedding into $\R^3$ as the neighborhood of an unknotted circle, there are links which are nontrivial in $S^1\times\R^2$, yet become unlinked when embedded into $\R^3$ via this embedding. In Section 3, we exhibit an infinite family of such links that have nontrivial Jones polynomial.  In a sense, these are the simplest type of links one could consider which satisfy the constraints.

\begin{prop}\label{JonesLinks}
There are infinitely many nontrivial 2-component links in $S^1\times\R^2$ that are unlinked when embedded into $\R^3$ but whose linking in $S^1 \times \mathbb{R}$ can be detected by the Jones polynomial.
\end{prop}

Following the example of Eliahou, Kauffman, and Thistlethwaite \cite{eliahou_infinite_2003}, in Section 4, we define for each four-ended tangle $T$ with a specific orientation (which we will call a left-right oriented tangle) an invariant $\con(T)$ related to the Conway polynomial. This invariant $\con(T)$ is an element of the free $\Z[z]$-module $\Z[z]^2$, and can be used to calculate the Jones and Conway polynomials of links depending on a tangle:

\begin{thm}\label{NewInvt}
Let $L$ be a diagrammatic operator which takes a single left-right oriented tangle $T$ and yields an oriented link $L(T)$. Then this operator induces a $\Z[z]$-module homomorphism $\varphi_L : \Z[z]^2\to \Z[z]$ such that $\nabla(L(T)) = \varphi_L(\con(T))$.


\end{thm}

\noindent We combine this theorem and results of Kauffman \cite{kauffman_conway_1981} to produce an explicit example indicating that the Conway polynomial does not detect causality in the given setting.

After examining the Jones polynomial in more detail, we then make the following conjecture:

\begin{conj}
Let $X$ be a $(2+1)$-dimensional globally hyperbolic spacetime with Cauchy surface $\Sigma$ not homeomorphic to $S^2$ or $\R P^2$. Then the Jones polynomial detects causality between any two events in $X$.
\end{conj}

\noindent \textit{Acknowledgements.}  Both authors would like to thank Vladimir Chernov and Ina Petkova for suggesting the problem and their guidance throughout the project. The first author would like to additionally thank Charles Livingston for a helpful conversation. The second author would like to thank Vanessa Pinney for her insight on enumeration of links. JS received support from NSF Grant DMS-1711100.

\section{Background} \label{sec:background}
In this section we give definitions and useful results related to the Conway polynomial, the Jones polynomial, and tangles.

Let $L^1$ be a link diagram with a distinguished crossing of the form $\xygraph{!{0;/r1.0pc/:}[u(0.5)]!{\xunderv}}\;$. Denote by $L^0$ and $L^{\infty}$ the resulting link diagrams arising from performing smoothing changes at the distinguished crossing, as in Figure \ref{fig:kauffman}.

\begin{figure}[ht]
\centering
\captionsetup{justification=centering}
\begin{tabular}{ccccc}
\xygraph{!{0;/r2.2pc/:}[u(1.1)]!{\xunderv}} & $\;\;\;\;$ &

 \includegraphics[width = 0.07\textwidth]{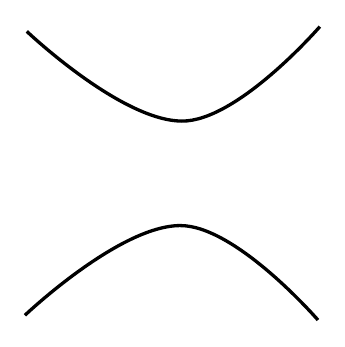} & $\;\;\;\;$ 
    & \includegraphics[width = 0.07\textwidth]{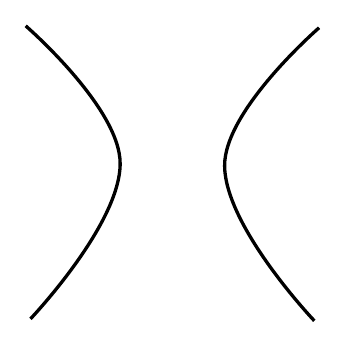} \\
    &&\\
 $L^1$ &$\;\;\;$ & $L^0$ & $\;\;\;$ & $L^\infty$\\
\end{tabular}
\caption{Local changes made near a crossing in a link $L$.}
\label{fig:kauffman}
\end{figure}

\begin{defn}\label{BracketDef}
The Kauffman bracket of a link $L$, $\langle L \rangle$, is the Laurent polynomial in a variable $A$ defined by
	\[ \langle \bigcirc\rangle = 1, 
		\quad \langle \bigcirc \cup L\rangle = (-A^2-A^{-2})\langle L\rangle, \]
	\[ \left\langle L^1 \right\rangle = A \left\langle L^0 \right\rangle + A^{-1} \left\langle L^\infty \right\rangle.
	\]
\end{defn}
\noindent The Kauffman bracket is not a link invariant; one must adjust for the writhe of the link diagram.
\begin{defn}\label{JonesDef}
The Kauffman polynomial of a link $L$, $K(L)$, is defined to be 
    \[K(L) := (-A^3)^{-w(L)}\langle L\rangle \]
where $w(L)$ is the writhe of the diagram.  The Jones polynomial $V(L)$ can be obtained by substituting $t^{-1/4}$ for $A$ in the Kauffman polynomial \cite[Thm 2.8]{kauffman_state_1987}.
\end{defn}
\begin{thm}{\cite[Thm 2.6]{kauffman_state_1987}}
The Kauffman polynomial (and, therefore, the Jones polynomial) is a link invariant.
\end{thm}

For an oriented link diagram, let $L_{+}, L_{-}, L_{0}$ be the resulting link diagrams arising from crossing and smoothing changes on a local region of a specified crossing of the diagram, as in Figure \ref{fig:Conway}.  

\begin{figure}[ht]
    \centering
    \includegraphics[width=0.3\textwidth]{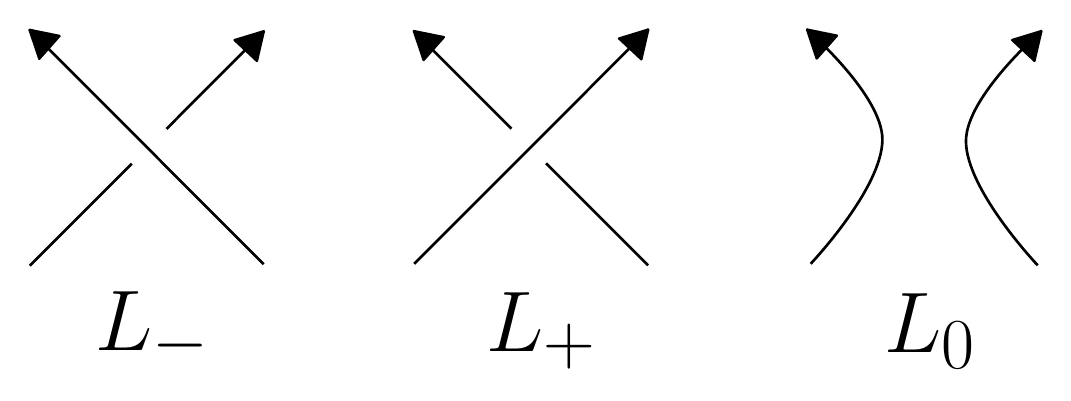}
    \caption{An oriented skein triple.  Each link differs from the link $L$ near a crossing.}
    \label{fig:Conway}
\end{figure}

\begin{defn}\label{ConwayPolynomialDef}
The Conway polynomial of $L$, $\nabla(L)$, is defined by the following skein relations:
    \[\nabla(O) = 1, \text{ where $O$ is any diagram of the unknot, and}\]
    \[\nabla(L_+)-\nabla(L_-) = z\nabla(L_0).\]
\end{defn}
\noindent The Conway polynomial of a link gives the Alexander polynomial via a change of variables.  Thus, the Conway polynomial is a link invariant \cite[Thm 3.4]{kauffman_conway_1981}.

We will be interested in computing the Jones and Conway polynomials of links containing specific tangles.

\begin{figure}[ht]
    \centering
    \includegraphics[width=0.15\textwidth]{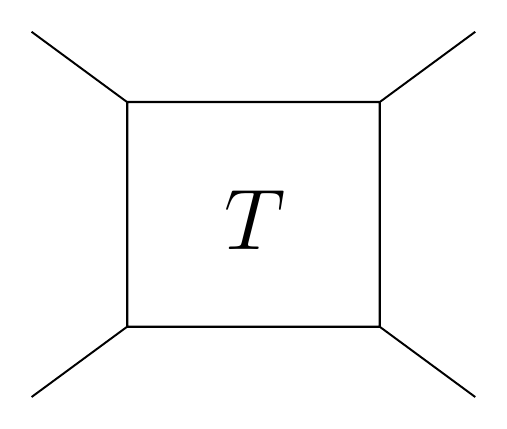}
    \caption{A tangle $T$.}
    \label{fig:tangle}
\end{figure}

\begin{defn}
A (2-string) tangle $T$ is a 3-dimensional ball with two strings properly embedded in it, that is, the endpoints of the strings are fixed on the boundary of a ball.
\end{defn}
\noindent Note that in some situations we may also allow for additional (knotted and linked) components inside the ball.

Choosing the fixed points on the boundary of the ball to be along the equator, one can arrange for the tangle to be in general position with respect to the disk bounded by the equator.  This allows for a diagrammatic representation of such a tangle; see Figure \ref{fig:tangle}. It will be useful to denote certain tangles by integers.  See Figure \ref{fig:integertangles}.

\begin{figure}[ht]
\centering
\captionsetup{justification=centering}
\begin{tabular}{ccccccccccc}
    \includegraphics[width = 0.06\textwidth]{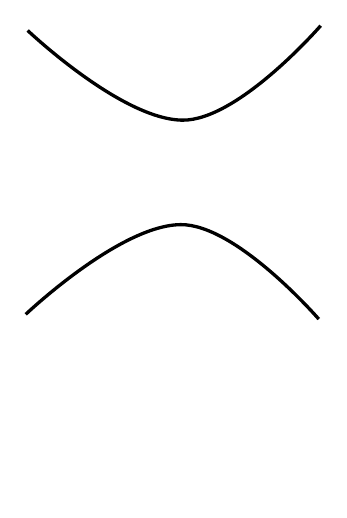} & $\;\;\;\;$ 
    & \includegraphics[width = 0.06\textwidth]{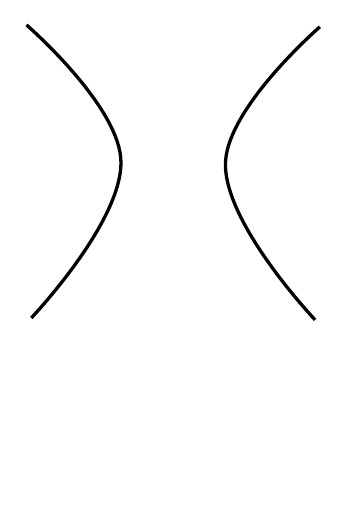} & $\;\;\;\;$ 
    & 
    \includegraphics[width = 0.17\textwidth]{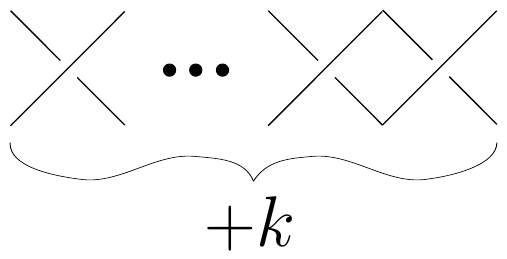} & $\;\;\;\;$ 
    & \includegraphics[width = 0.17\textwidth]{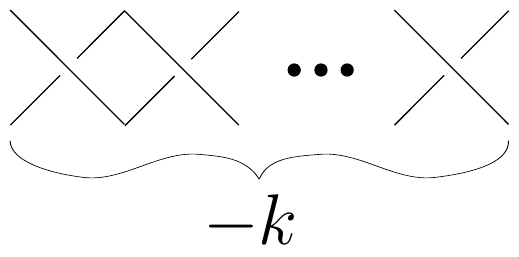}\\
    &&&\\
(a) Tangle $0$ & $\;\;\;$ & (b) Tangle $\infty$ & $\;\;\;$ & (c) Tangle $k$ & $\;\;\;$ & (d) Tangle $-k$ \\
\end{tabular}
\caption{Integer tangles.}
\label{fig:integertangles}
\end{figure}

Given tangles $T$ and $U$, we denote by
\begin{itemize}
    \item $-T$ the reflection of $T$ with respect to the equatorial disk,
    \item $T^\rho$ the reflection of $T$ across the NW-SE axis (see Figure \ref{fig:rho}),
    \item $T^N$ the numerator closure of $T$,
    \item $T^D$ the denominator closure of $T$,
    \item $T+U$ the tangle sum of $T$ with $U$, and
    \item $T*U$ the tangle ``vertical" sum of $T$ with $U$.
\end{itemize}  
See Figure \ref{fig:tangleops} for schematic pictures of the latter four operations. 
\begin{figure}[ht]
\centering
\captionsetup{justification=centering}
\begin{tabular}{ccc}
     \includegraphics[width = 0.09\textwidth]{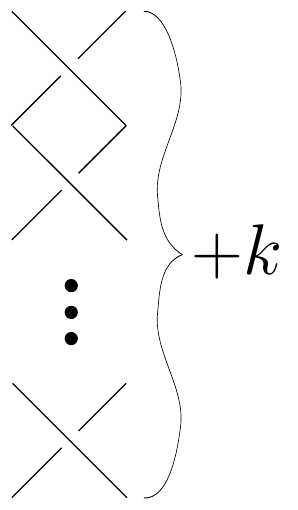} &
    $\;\;\;\;$ 
    & \includegraphics[width = 0.09\textwidth]{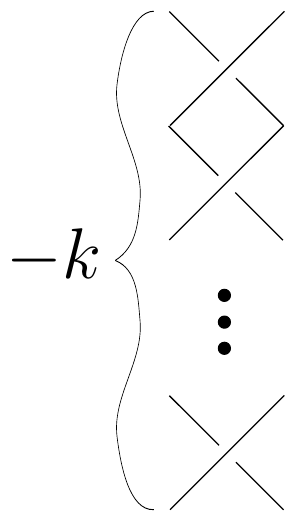}\\
    &&\\
 Tangle $k^\rho$ & $\;\;\;$ &  Tangle $(-k)^\rho$\\
\end{tabular}
\caption{The reflections of tangles $k$ and $-k$ reflected across the NW-SE axis.}
\label{fig:rho}
\end{figure}
\begin{figure}[ht]
\centering
\captionsetup{justification=centering}
\begin{tabular}{ccccccc}
\includegraphics[width = 0.11\textwidth]{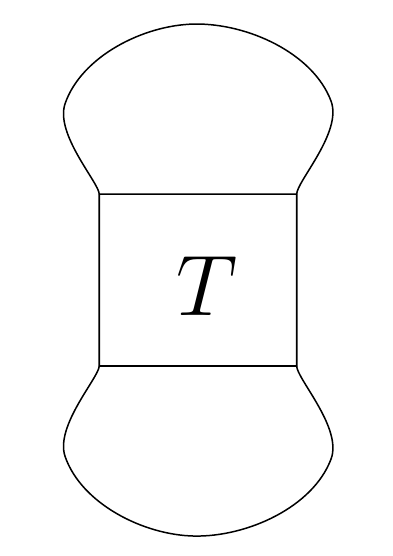} & $\;\;\;\;$ 
    & \includegraphics[width = 0.17\textwidth]{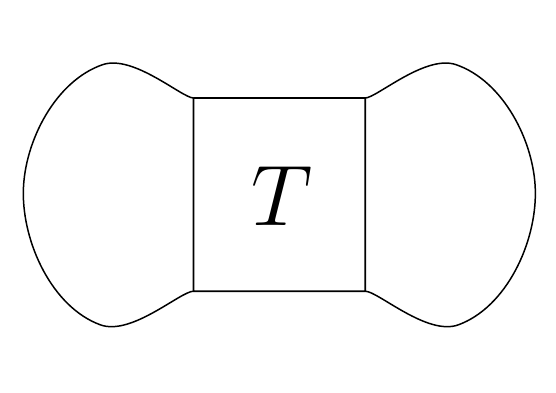} & $\;\;\;\;$ 
    & \includegraphics[trim={0 8mm 0 8mm},width = 0.25\textwidth]{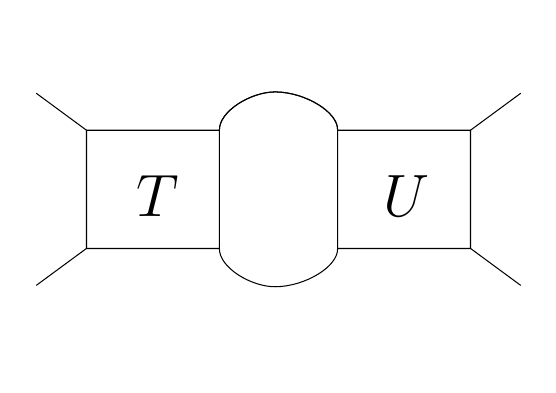} & $\;\;\;\;$ 
    & \includegraphics[trim={8mm 0 8mm 0},width = 0.10\textwidth]{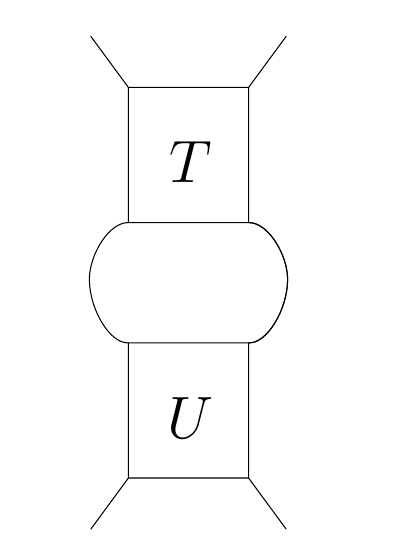} \\
    &&&&&&\\
(a) $T^N$ & $\;\;\;$ &(b) $T^D$ & $\;\;\;$ &(c) $T+U$ & $\;\;\;$ &(d) $T*U$\\
\end{tabular}
\caption{Tangle operations.}
\label{fig:tangleops}
\end{figure}

We denote by $L^T$  a link diagram which, in a disk intersecting the link diagram at four points, contains the tangle diagram $T$.
For example, $T^N$ and $T^D$ are such link diagrams. If $L^T$ is such a link diagram with $T$ contained in some disk $D$, then performing local moves (such as crossing changes and smoothings) in $T$ results in a new tangle diagram $T' \subset D$ and a new link diagram $L^{T'}$ which is unchanged outside of $D$.  With this notation, the bracket polynomial $\langle L^T \rangle$ can be formally expanded to $f(T)\langle L^0\rangle + g(T) \langle L^\infty \rangle$ where $f(T)$ and $g(T)$ are Laurent polynomials in $\mathbb{Z}[A, A^{-1}]$ which depend only on the tangle diagram $T$.

\begin{defn}[\cite{eliahou_infinite_2003}]\label{BracketVectorDef}
Let $T$ be a tangle.  Define the bracket vector of $T$ to be
    \[br(T):= \begin{bmatrix} f(T) & g(T) \end{bmatrix}^t. \]
\end{defn}
\noindent With the notation from the previous paragraph, the definition of $\br(T)$ gives $$\langle L^T \rangle = f(T)\langle L^0\rangle + g(T)\langle L^\infty\rangle = [\langle L^0\rangle\ \langle L^\infty\rangle]\br(T).$$

\begin{prop}[\cite{eliahou_infinite_2003}]\label{BracketVectorSumProp}  Let $T$ and $U$ be tangles.  Then we have the
formula
	\begin{equation*}\label{BracketSum}
	\br(T + U) = \begin{bmatrix}f(U)&0\\g(U)&f(U) + \delta g(U)\end{bmatrix}\br(T),
	\end{equation*}
where $\delta = -A^2-A^{-2}$.
\end{prop}

\noindent In Section \ref{sec:conway}, we define a similar invariant for computing the  Conway polynomial.  

\section{A Family of Links With Nontrivial Jones Polynomial}

A 2-component link $(N,K_1\sqcup K_2)$ in the solid torus $N$ can equivalently be thought of as a 3-component link $(S^3,K_1\sqcup K_2\sqcup\mu)$ in $S^3$, where the additional component $\mu$ is a meridian of $N$ \cite[proof of Thm 1]{chernov_khovanov_2020}. In this setting, where we are concerned with skies of causally related events, the meridian $\mu$ may be oriented in either direction, while $K_1$ and $K_2$ must be oriented coherently with a longitude $\lambda$. In addition, we require that each of $K_1$ and $K_2$ be independently isotopic to $\lambda$, so that the linking numbers $\lk(K_1,\mu),\lk(K_2,\mu)$ are either both $1$ or both $-1$. We call such 3-component links \emph{2-sky-like}, and all links we will consider will be of this form.

\begin{rmk}
While all links arising from pairs of skies in our setting are 2-sky-like, it is possible that there are 2-sky-like links that don't arise as pairs of skies.
\end{rmk}

We consider the family of link diagrams $C(T)$ depicted in Figure \ref{fig:C(T)}(a), where $T$ can be any tangle. When the tangle $T$ has a diagram which can be appropriately oriented, we also will consider $C(T)$ with two specific orientations: $C_+(T)$ and $C_-(T)$ as in Figure \ref{fig:C(T)}(b) and (c). This is a natural family to consider since the simplest examples of 2-sky-like links have geometric intersection number two with the meridian. Note that if $C(T)$ is 2-sky-like, then the 2-component link where $\mu$ is ignored is the numerator closure of $T$.

\begin{figure}[ht]
\centering
\captionsetup{justification=centering}
\begin{tabular}{ccccc}
\includegraphics[width = 0.2\textwidth]{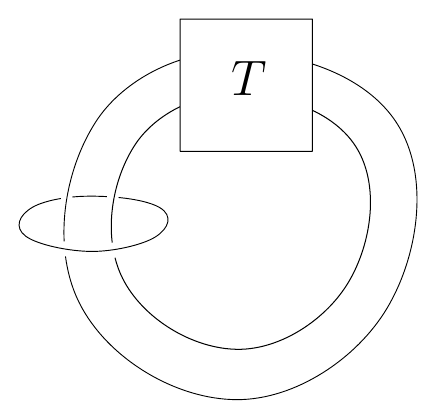} & $\;\;\;\;$ 
    & \includegraphics[width = 0.2\textwidth]{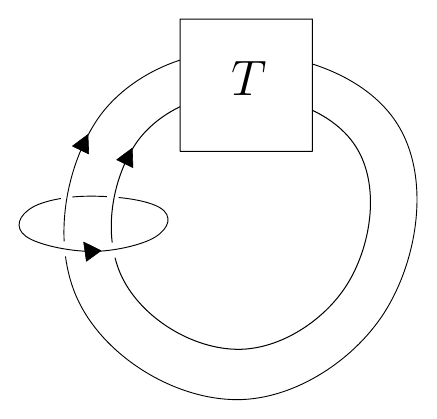} & $\;\;\;\;$ 
    & \includegraphics[width = 0.2\textwidth]{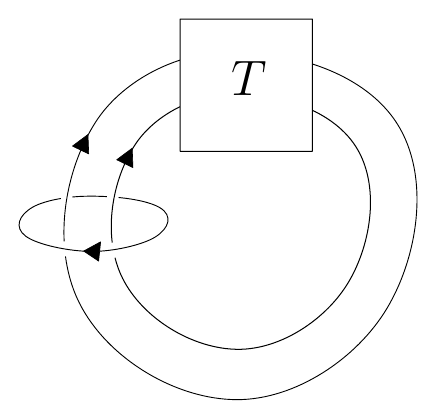} \\
(a) $C(T)$ & $\;\;\;$ &(b) $C_+(T)$ & $\;\;\;$ &(c)  $C_-(T)$ \\
\end{tabular}
\caption{The unoriented three-component link $C(T)$ and two specific orientations $C_+(T)$ and $C_-(T)$.}
\label{fig:C(T)}
\end{figure}

Let $H$ be the connected sum of two Hopf links, i.e. $C(0)$.
With the orientations as in $C_+(0)$ and $C_-(0)$, $H$ is 2-sky-like.  In fact, $H$ is a pair of skies in the following sense. Let $X$ is a $(2+1)$-dimensional globally hyperbolic spacetime with Cauchy surface homeomorphic to $\R^2$. The solid torus $N$ is the space of future-directed null geodesics (light rays) in $X$, and is homeomorphic to a solid torus. Then $H$ corresponds to an unlink of two longitudinal loops in the solid torus $N$, which corresponds to the skies of a pair of causality unrelated points in $X$.

\newpage

Using Definition \ref{BracketDef}, we calculate the bracket vector of $C(T)$:

\begin{prop}\label{C(T)NumDenom}
    \[ \langle C(T) \rangle =
        (-A^6-A^{-6})\langle T^N\rangle + (-A^4-A^{-4}+2)\langle T^D\rangle. \]
\begin{proof}
Apply Definition \ref{BracketDef} iteratively to each crossing in the diagram of $C(T)$.
\end{proof}
\end{prop}

Before we apply Proposition \ref{C(T)NumDenom} to show Theorem \ref{U(n,m)}, we lay out some facts about the Kauffman bracket.

\begin{lem}\label{BracketCalculations}
Let $L$ be a link diagram, and let $L_{t+}$ and $L_{t-}$ be the result of adding a positive kink and a negative kink, respectively, to $L$ via a type I Reidemeister move. Then
\begin{enumerate}
    \item $\langle L_{t+}\rangle = -A^3\langle L\rangle.$
    \item $\langle L_{t-}\rangle = -A^{-3}\langle L\rangle.$
    \item For a positive integer $k$,
        \[ \langle (k^\rho)^D\rangle = \langle k^N\rangle = -A^{-k+2} + \sum_{j=0}^k (-1)^{j+1}A^{4j-k-2}. \]
    \item For a negative integer $-k$,
        \[ \langle ((-k)^\rho)^D\rangle= -A^{k-2} + \sum_{j=0}^k (-1)^{j+1}A^{-4j+k+2}.\]
    \item If $L'$ is another link diagram, then
        \[ \langle L\#L'\rangle = \langle L\rangle\langle L'\rangle. \]
\end{enumerate}
\end{lem}

\begin{proof}
To see (1) and (2) apply Definition \ref{BracketDef} to $L_{t+}$ and $L_{t-}$, respectively; see, for example, \cite[Prop 2.5]{kauffman_state_1987}. To prove (3), we apply Definition \ref{BracketDef} at any crossing of $k^N$.  Combining this with (1) and (2), we obtain the recursive relation
    \[ \langle (k^\rho)^D\rangle = A(-A^3)^{k-1} + A^{-1}\langle((k-1)^\rho)^D\rangle. \]
Part (3) follows by induction on $k$. Part (4) follows immediately from the fact that the Kauffman bracket of a mirror diagram is obtained by substituting $A^{-1}$ for $A$ 
\cite[Prop 2.7]{kauffman_state_1987}.
Finally, note that that $V(L\#L') = V(L)V(L')$ and
$w(L\#L') = w(L) + w(L')$ for any appropriate orientation on $L$ and $L'$. Then (5) follows easily from Definition \ref{JonesDef}.
\end{proof}

\begin{defn}
A \emph{regular isotopy} between link diagrams is a series of type II or type III Reidemeister moves.
A \emph{balanced isotopy} between link diagrams is a series of type II Reidemeister moves, type III Reidemeister moves, or pairs of type I Reidemeister moves that preserve the writhe of the diagram for any orientation, e.g a positive kink and a negative kink. Two link diagrams are said to be \emph{balanced-isotopic} if they are related by a balanced isotopy.
\end{defn}
\begin{cor}
The bracket polynomial is an invariant of regular-isotopic and balanced-isotopic link diagrams.
\begin{proof}
This is a direct consequence of Definition \ref{JonesDef} and Lemma \ref{BracketCalculations}.
\end{proof}
\end{cor}

We now consider the family of 2-sky-like links $U(n,m) = C(TU(n,m))$ (see Figure \ref{fig:Unm}), where $TU(n,m)$ is as in Figure \ref{fig:TUnm} for integers $n$ and $m$.

\begin{figure}
    \centering
    \includegraphics[width=0.3\textwidth]{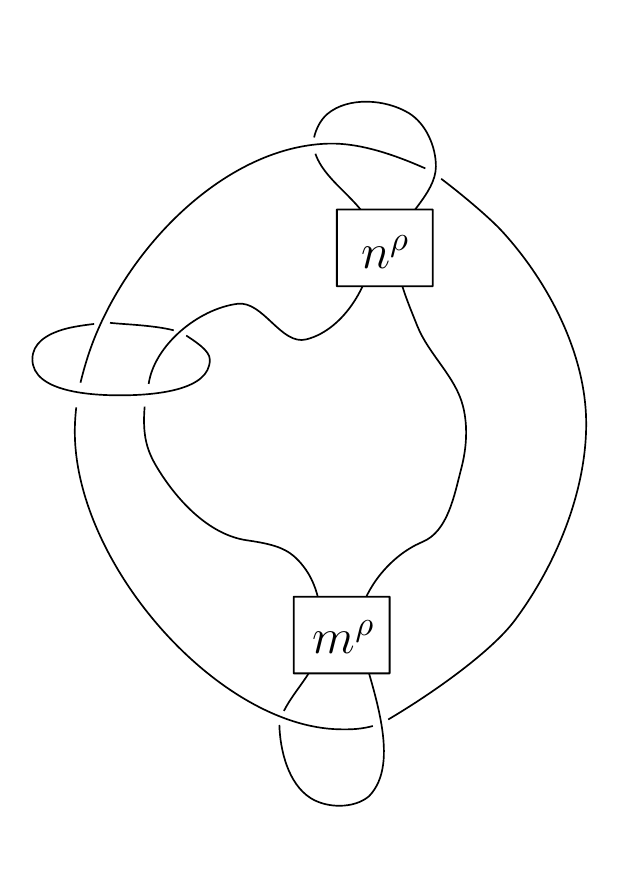}
    \caption{$U(n,m)$.}
    \label{fig:Unm}
\end{figure}

\begin{figure}
    \centering
    \includegraphics[width=0.3\textwidth]{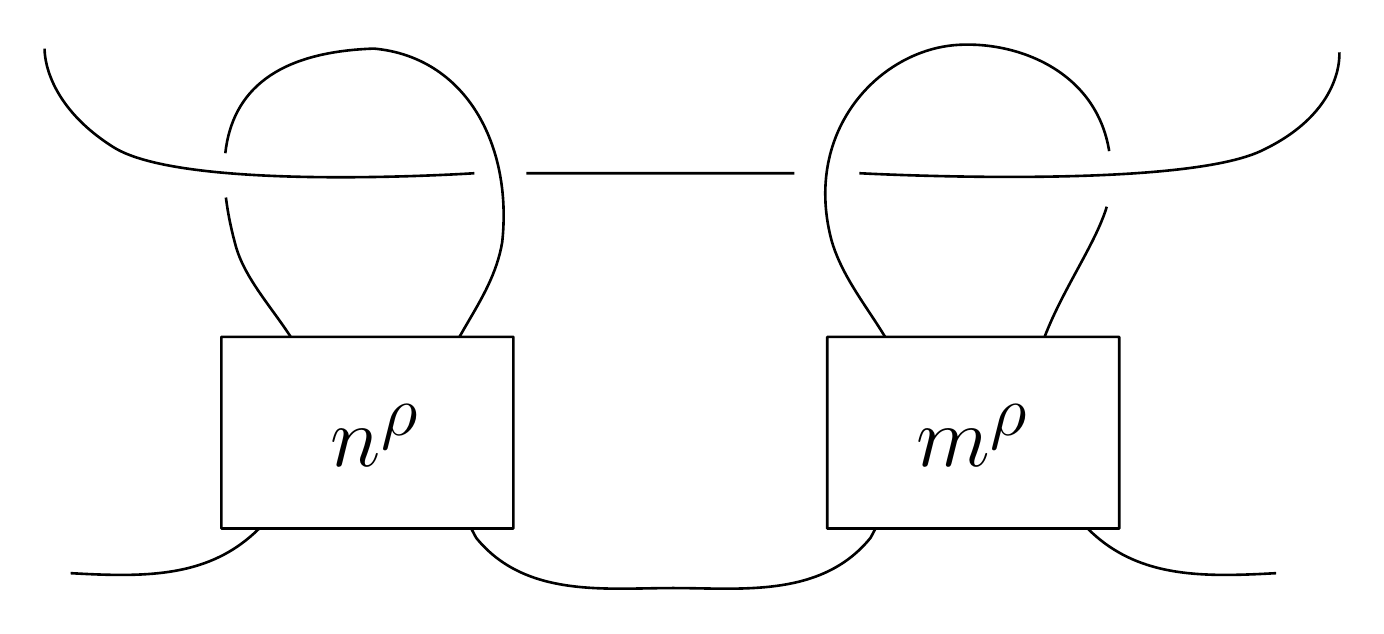}
    \caption{$TU(n,m)$.}
    \label{fig:TUnm}
\end{figure}

\begin{defn}  Let $y(x)$ be a nonzero Laurent polynomial.  Denote by $h(y)$ the highest exponent of $x$ and $\ell(y)$ the lowest exponent of $x$ in $y(x)$.  The span of a nonzero Laurent polynomial $y(x)$, denoted  $\mathrm{span}(y(x))$, is defined to be $h(y)-\ell(y)$.
\end{defn}
Clearly, $\mathrm{span}(y_1y_2) = \mathrm{span}(y_1) + \mathrm{span}(y_2)$. We will use this fact in the next theorem.
\begin{thm}\label{U(n,m)}
For $(n,m) \neq 0$, $V(U(n,m)) \neq V(H)$ for any orientation of $U(n,m)$ and $H.$
\begin{proof}
From Definition \ref{JonesDef}, we have that, for any link diagram $L$,
    \[ \mathrm{span}(V(L)) = \frac{1}{4}\mathrm{span}(\langle L\rangle). \]
The Jones polynomial of the Hopf link is $-t^{5/2}-t^{1/2}$ for some orientation \cite{linkinfo},
so $V(H) = t^5 + 2t^3 + t$ for some orientation.  It follows that $$\mathrm{span}(\langle H \rangle) = 4\cdot\mathrm{span}(V(L)) = 4\cdot 4 = 16.$$
Thus it suffices to show that $\mathrm{span}(\langle U(n,m)\rangle) \neq 16$ for $(n,m)\neq (0,0)$.

The mirror image of $U(n,m)$ is balanced-isotopic to $U(-n-1,-m+1)$, and
by flipping the link across the horizontal axis, we see that $U(n,m)$ is isotopic to $U(m-1,n+1)$ for all integers $n,m$. The composition of flipping and mirroring takes $U(n,m)$ to $U(-m,-n)$.  Thus, it is sufficient to consider $n \ge 0$ and $m \ge -n$.

First, we consider $U(n,-n)$ for $n$ positive. Note that $TU(n,-n)^N$
is regularly isotopic to a split union of two unknots.
By Proposition \ref{C(T)NumDenom}, we have
	\[ \langle U(n,-n)\rangle = (-A^6-A^{-6})(-A^2-A^{-2}) 
		+ (-A^4-A^{-4}+2)\langle TU(n,-n)^D\rangle. \]
Applying \cite[Theorem 5.7]{natario_linking_2004} and Lemma \ref{BracketCalculations}, we see that
    \begin{align*}
        \langle TU(n,-n)^D\rangle &= \left[A^{-n-6} + (A^4+A^{-4})\sum_{k=1}^n(-1)^kA^{4k-n-2}\right] \\
        & \;\;\;\;\;	\times\left[A^{n+6} + (A^4+A^{-4})\sum_{k=1}^n(-1)^kA^{-4k+n+2}\right]. 
    \end{align*}
Therefore, we have
\begin{align*}
    h(\langle TU(n,-n)^D\rangle) &= ((4n-n
    -2)+4)+(n+6) = 4n+8,\\
    \ell(\langle TU(n,-n)^D\rangle) &= ((-4n+n+2)-4)+(-n-6) =  -4n-8.
\end{align*}
Thus
\begin{align*}
    h(\langle U(n,-n)\rangle) &= \max\{8, (4n+8)+4\} = 4n+12,\\
    \ell(\langle U(n,-n)\rangle) &= \min\{-8, (-4n-8)-4\} = -4n-12,
\end{align*}
and so $\mathrm{span}(\langle U(n,-n)\rangle) = 4n+12-(-4n-12) \ge 24 > 16$.
    

\begin{figure}
    \centering
    \includegraphics[width=0.8\textwidth]{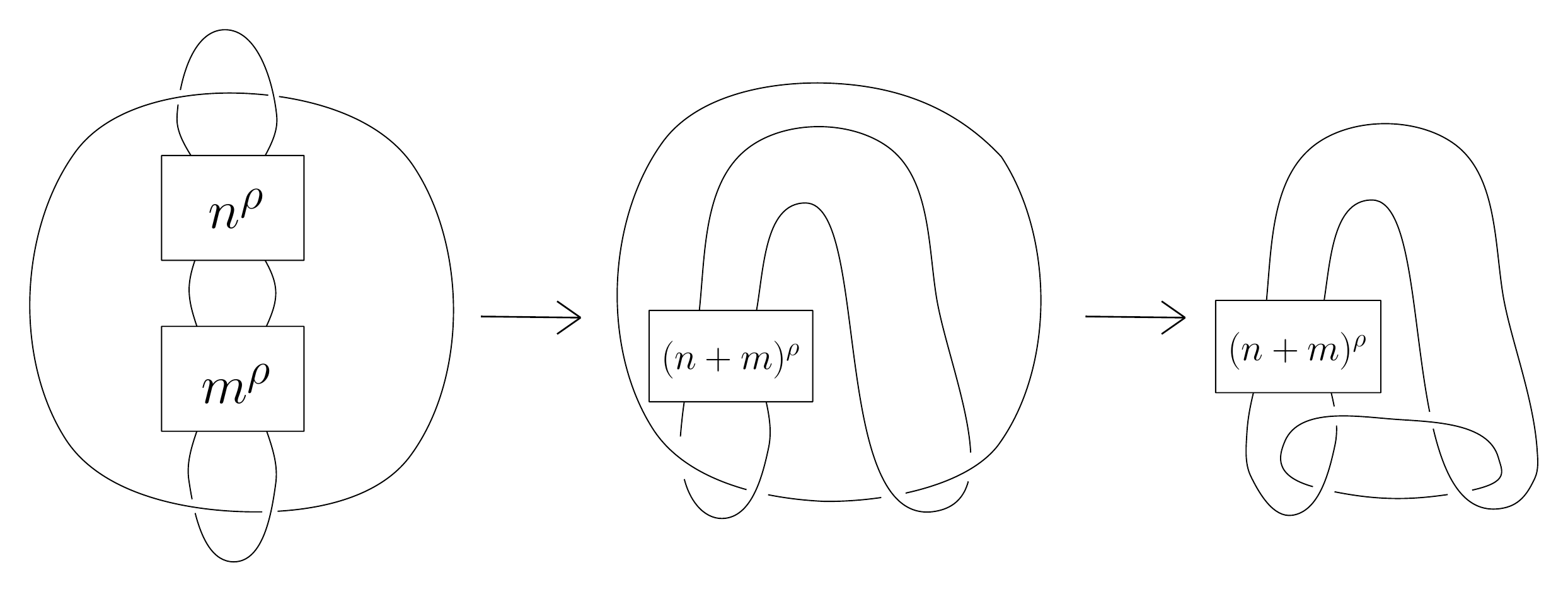}
    \caption{A regular isotopy.}
    \label{fig:regisotopy}
\end{figure}

Next, we consider $U(n,-m)$, with $0\le m < n$. In this case, $TU(n,-m)^N$ is no longer a trivial link. However, observe that after performing the regular isotopy shown in  Figure \ref{fig:regisotopy} we obtain a link which is balanced isotopic to $C((n+m)^\rho)$. Therefore, we can apply Proposition \ref{C(T)NumDenom}.
Using this and Lemma \ref{BracketCalculations}, we see that 
	\begin{align*} 
	\langle TU(n,-m)^N\rangle &= (-A^6-A^{-6})\langle (n-m)^{\rho N}\rangle
		+ (-A^4-A^{-4}+2)\langle(n-m)^{\rho D}\rangle \\
	&= (-A^6-A^{-6})(-A^3)^{n-m} \\&\quad+ (-A^4-A^{-4}+2)\left(-A^{-(n-m)+2} 
		+ \sum_{k=0}^{n-m} (-1)^{k+1}A^{4k-(n-m)-2}\right).
	\end{align*}
Applying \cite[Theorem 5.7]{natario_linking_2004} and Lemma \ref{BracketCalculations}, we calculate that
	\[ \langle TU(n,-m)^D\rangle 
		= \left[A^{-n-6} + (A^4+A^{-4})\sum_{k=1}^n(-1)^kA^{4k-n-2}\right] \]
	\[ \times \left[A^{m+6} + (A^4+A^{-4})\sum_{k=1}^m(-1)^kA^{-4k+m+2}\right]. \]
(Here, the empty sum when $m = 0$ is taken to be 0.)
This allows for a full calculation of $\langle U(n,-m)\rangle$, as Proposition \ref{C(T)NumDenom} gives
	\[ \langle U(n,-m)\rangle = (-A^6-A^{-6})\langle TU(n,-m)^N\rangle 
		+ (-A^4-A^{-4}+2)\langle TU(n,-m)^D\rangle. \]
After checking that $U(1,0)$ does not have trivial Jones polynomial, we consider the few remaining cases.  We omit many of the computations, as they are very similar to the computation of $\mathrm{span}(\langle U(n,-n)\rangle)$ above.
\begin{itemize}
\item If $m = 0$, and $n > 1$, we compute that 
    \[ \mathrm{span}(\langle U(n,0)\rangle) = (3n+8)-(-n-12) = 4n+20 > 16. \]
    Note that in this case, the highest terms in the two summands cancel.

\item If $n-m = 1$, then $n=m+1$ and $-m < 0$. In this case, there is cancellation in the sum in the formula for $\langle TU(n,-m)^N\rangle$.  We see that
\[ \mathrm{span}(\langle U(m+1,-m)\rangle) = (4m+15)-(-4m-13) = 8m+28 > 16. \]
\item If $n-m > 1$ and $-m < 0$, we compute that
\[ \mathrm{span}(\langle U(n,-m)\rangle) = (3n+m+8)-(-n-3m-8) = 4n+4m+16 > 16. \]
\end{itemize}

Finally, we consider the case of $U(n,m)$ with $n\ge 0$ and $m > 0$. We can always arrange that $n < m$, for if $n \ge m$ we can instead consider $U(m-1,n+1)$ (obtained by flipping $U(n,m)$ over) and we have $m-1 <n+1$. Because $U(0,1)$ is isotopic to the link L8n3 in \cite{linkinfo}, which has Jones polynomial not equal to that of $H$, we can further restrict to $m \ge 2$. Now
	\begin{align*} 
	\langle TU(n,m)^N\rangle &= (-A^6-A^{-6})\langle(n+m)^{\rho N}\rangle
		+ (-A^4-A^{-4}+2)\langle(n+m)^{\rho D}\rangle \\
	&= (-A^6-A^{-6})(-A^3)^{n+m} \\&\quad+ (-A^4-A^{-4}+2)\left(-A^{-(n+m)+2} 
		+ \sum_{k=0}^{n+m} (-1)^{k+1}A^{4k-(n+m)-2}\right)
	\end{align*}
and
	\begin{align*}
	 \langle TU(n,m)^D\rangle &= 
		\left[A^{-n-6} + (A^4+A^{-4})\sum_{k=1}^n(-1)^kA^{4k-n-2}\right] \\
	&\quad \; \times\left[A^{-m-5} + (A^4+A^{-4})\sum_{k=1}^{m-1}(-1)^kA^{4k-m-1}\right].  
	\end{align*}
First, the case where $n = 0$. Then we have
    \[ h(\langle TU(0,m)^N\rangle) = 3m+6 \text{ and } \ell(\langle TU(0,m)^N\rangle) = -m-6,\] 
    \[ h(\langle TU(0,m)^D\rangle) = 3m-7 \text{ and } \ell(\langle TU(0,m)^D\rangle) = -m-11,\] 
and so
    \begin{align*}
    \mathrm{span}(\langle U(0,m)\rangle) &= \max\{(3m+6)+6, (3m-7)+4\}-\min\{(-m-6)-6, (-m-11)-4\}\\
    &= (3m+12) - (-m-15) = 4m+27>16.
    \end{align*}
Now we consider the last remaining case, 
$m > n > 0$:
    \[ h(\langle TU(n,m)^N\rangle) = 3(n+m)+6 \text{ and } \ell(\langle TU(n,m)^N\rangle) = -(n+m)-6,\] 
    \[ h(\langle TU(n,m)^D\rangle) = 3(n+m)+1 \text{ and } \ell(\langle TU(n,m)^D\rangle) = -(n+m)-11,\] 
and so
    \begin{align*}
    \mathrm{span}(\langle U(n,m)\rangle) &= \max\{(3(n+m)+6)+6, (3(n+m)+1)+4\}\\ &\quad-\min\{(-(n+m)-6)-6, (-(n+m)-11)-4\}\\
    &= (3(n+m)+12) - (-(n+m)-15) = 4(n+m)+27>16.
    \end{align*}
Thus $U(n,m)$ doesn't give trivial Jones polynomial for any integers $(n,m)\neq(0,0)$.
\end{proof}
\end{thm}

Proposition 1.1 follows from Theorem \ref{U(n,m)}. Namely, consider $U(n,-n)$. The corresponding 2-component link $TU(n,-n)^N$ is evidently an unlink. On the other hand, $U(n,-n)$ and the connected sum of two Hopf links $H$ have different Jones polynomials for any orientation of the components by Theorem \ref{U(n,m)}. Nat\'ario and Tod \cite[Cor 5.12]{natario_linking_2004} had already found an infinite family of links with nontrivial Jones polynomial, but did not consider links in the solid torus that were unlinked when embedded into $\R^3$.

Can the Jones polynomial detect causality in the given setting? In other words, are all 2-sky-like links distinguishable from the connected sum of two Hopf links $H$ by their Jones polynomials? To answer this question, it is necessary to ask if any non-trivial 2-sky-like link has the same bracket polynomial as $H=C(0)$ up to multiplication by a power of $A$. We first check LinkInfo \cite{linkinfo} and find no 3-component links up to 11 crossings with the same Jones polynomial as the connected sum of two Hopf links. Next, we try to find properties of a link with this Jones polynomial. For simplicity, we still consider the family $C(T)$. One can calculate, using Definition \ref{BracketDef}, that
    \[ \langle C(0)\rangle = A^8+2+A^{-8}\qquad\text{and}\qquad
        \langle C(\infty)\rangle = -A^2-A^{-2}. \]
Then by Definition \ref{BracketVectorDef}, for any tangle $T$,
    \[ \langle C(T)\rangle = \begin{bmatrix}
        A^8+2+A^{-8}&-A^2-A^{-2}\end{bmatrix}\br(T). \]
If, for some orientation of $C(T)$, $V(C(T)) = V(C(0))$, then
$\langle C(T)\rangle = (-A^3)^n\langle C(0)\rangle$ for some integer $n$. 
One way we could hope to get this is by finding $T$ such that $\br(T) = (-A^3)^n\br(0) = [(-A^3)^n\ 0]^t$. However, we have the following theorem:

\begin{thm}\cite[Thm 13]{sikora2020tangle}\label{JonesConjThm}
There exists a nontrivial knot with Jones polynomial 1 if and only if there exists a tangle $T \neq 0$ such that $\br(T)=[rA^n\ 0]^t$ for some integers $r,n\in\Z$.
\end{thm}

\noindent It remains an open problem whether there exists a nontrivial knot with trivial Jones polynomial \cite{dasbach_does_1997}. We also do not know that $\langle C(T)\rangle = (-A^3)^n\langle C(0)\rangle$ necessarily implies $\br(T)=(-A^3)^n\br(0)$. There are certainly many vectors $V\in\Z[A,A^{-1}]^2$ that satisfy
    \[ \begin{bmatrix}
        A^8+2+A^{-8}&-A^2-A^{-2}\end{bmatrix}V = A^8+2+A^{-8}, \]
but we do not know if these vectors arise as bracket vectors of tangles. We pose this as an open question:

\begin{question}\label{BracketVectorQuestion}
Does
    \[ \begin{bmatrix}
        A^8+2+A^{-8}&-A^2-A^{-2}\end{bmatrix}\br(T) = \br(0) = [1\ 0]^t \]
imply $\br(T)=\br(0)$?
\end{question}

If Question \ref{BracketVectorQuestion} is answered in the positive, then whether or not the Jones polynomial detects causality in our setting is tied to whether or not the Jones polynomial detects the unknot by Theorem \ref{JonesConjThm}. If Question \ref{BracketVectorQuestion} is answered in the negative, then there must exist some tangle $T$ that provides a counterexample, and this tangle will necessarily be nontrivial, but not necessarily 2-sky-like.  Our next step would be to check if it is 2-sky-like, and then whether it appears as a pair of skies. Since there is no clear answer either way, we pose the following conjecture:

\begin{conj}[Conjecture 1.1]
Let $X$ be a $(2+1)$-dimensional globally hyperbolic spacetime with Cauchy surface $\Sigma$ not homeomorphic to $S^2$ or $\R P^2$. Then the Jones polynomial detects causality between any two events in $X$.
\end{conj}

\begin{rmk}
Not all 2-sky-like links are of the form $C(T)$. For a 2-sky-like link to be $C(T)$ for some tangle $T$, it is necessary and sufficient that the corresponding 2-component link in the solid torus has some meridional disk that meets each component exactly once. An example of a 2-sky-like link that does not seem to satisfy this condition is given in Figure \ref{fig:counterexample}.
\end{rmk}

\begin{figure}
    \centering
    \includegraphics[width=3in]{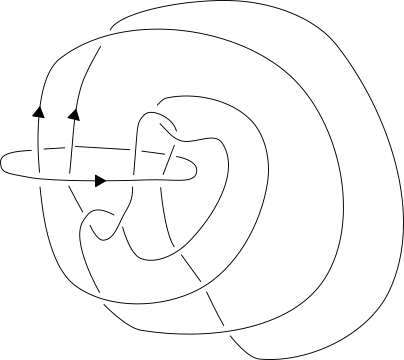}
    \caption{A 2-sky-like link not of the form $C(T)$.}
    \label{fig:counterexample}
\end{figure}

\section{Tangle Invariants Related to the Conway Polynomial} \label{sec:conway}

We now investigate the Conway polynomial and causality by noting some tangle invariants. Of particular importance is the way that tangle invariants are related to certain link invariants. We first distinguish two classes of tangles.
\begin{defn}\label{defn:oriented tangle}
An oriented tangle $T$ is \emph{left-right oriented} if its ends are oriented as in Figure \ref{fig:tangleori}(a). 
If $T$ is oriented as in Figure \ref{fig:tangleori}(b) or (c), we say that $T$ is \emph{diagonally oriented}.
\end{defn}
\begin{figure}[ht]
\centering
\captionsetup{justification=centering}
\begin{tabular}{c}
\includegraphics[width = 0.5\textwidth]{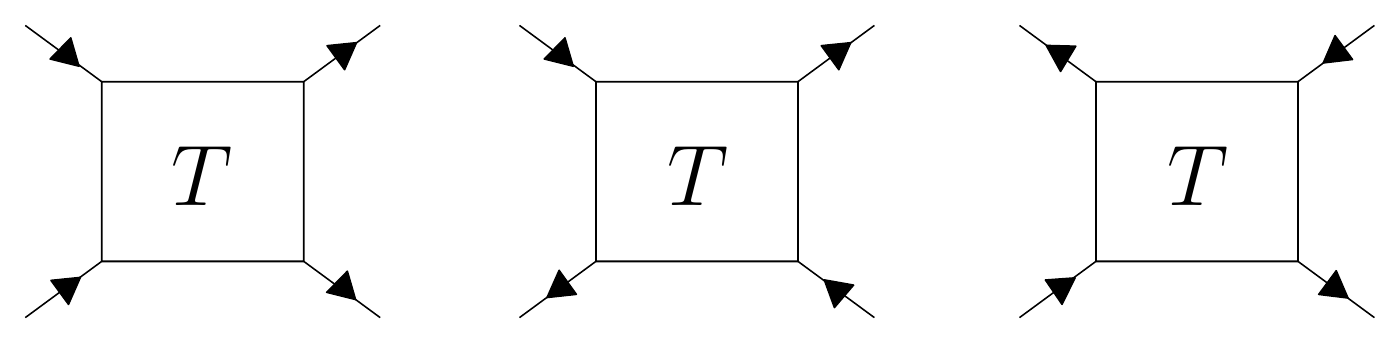}  \\
(a)  \hspace{.85in} (b) \hspace{.85in} (c) \\
\end{tabular}
\caption{(a) A left-right oriented tangle.  (b) \& (c) Diagonally oriented tangles.}
\label{fig:tangleori}
\end{figure}
\noindent Note that while a diagonally oriented tangle can have consistently oriented numerator and denominator closures, a left-right oriented tangle only has a consistently oriented numerator closure.  It is also important to note that the tangles $0$ and $\infty$ can be diagonally oriented, but only $0$ can be left-right oriented.

For diagonally oriented tangles, Kauffman \cite{kauffman_conway_1981} gave a definition for an invariant which is well-behaved under tangle addition.

\begin{defn}[\cite{kauffman_conway_1981}]
Given a diagonally oriented tangle $T$, let the fraction of the tangle be denoted $F(T) = \nabla(T^N)/\nabla(T^D)$.
\end{defn}
\noindent We will usually not consider $a/b$ equivalent to $ac/bc$, so that the numerator and the denominator of the fraction are well-defined quantities.
\begin{thm}[\cite{kauffman_conway_1981}, originally due to Conway]\label{TangleFractionSumThm}
Given tangles $T$ and $U$ that are diagonally oriented such that $T+U$ is coherently oriented, $F(T+U) = F(T) + F(U)$. Specifically,
    \[ \nabla((T+U)^N) = \nabla(T^N)\nabla(U^D) + \nabla(T^D)\nabla(U^N) \]
and
    \[ \nabla((T+U)^D) = \nabla(T^D)\nabla(U^D). \]
\end{thm}

Recall that we denote by $L^T$ a link diagram which, inside some disk intersecting the link diagram in four points, contains a tangle diagram $T$.  When $L^T$ is oriented, then $T$ is also oriented and the resulting orientation may be left-right or diagonal as in Definition \ref{defn:oriented tangle}.
The following fact will be useful.

\begin{thm}\label{ConReduceThm}
Let $L^T$ be an oriented link diagram which contains a left-right oriented tangle diagram $T$.
Then Definition \ref{ConwayPolynomialDef} can be used to calculate $\nabla(L^T)$ as a linear combination
    \[ \nabla(L^T) = p(T)\nabla(L^0) + q(T)\nabla(L^1) \]
for some $p(T),q(T) \in \Z[z]$ depending only on the tangle diagram $T$, where $0$ and $1$ denote the corresponding tangles. \\ Similarly, if $L^T$ is an oriented link diagram  which contains a diagonally oriented tangle diagram $T$, then $\nabla(L^T)$ can be calculated as
    \[ \nabla(L^T) = P(T)\nabla(L^0) + Q(T)\nabla(L^\infty) \]
for some $P(T),Q(T)\in\Z[z]$ depending only on the tangle diagram  $T$.
\end{thm}
\begin{proof}
We first assume that all tangles are left-right oriented unless otherwise stated. We induct on the number of crossings $n$ of a tangle $T$. For $n=0$, $T$ must be the (split) union of the tangle 0 and possibly some finite number of simple loops. If the number of loops is nonzero, $\nabla(L^T) = 0$ since the Conway polynomial of a split link is $0$. Otherwise, $\nabla(L^T) = \nabla(L^0)$.

Let $T$ be a left-right oriented tangle with $n=k$ crossings. Suppose that $\nabla(L^T)$ can be formally expanded as a linear combination of $\nabla(L^0)$ and $\nabla(L^1)$ for some $k \ge 1$. We show that if $T$ has $n=k+1$ crossings, $\nabla(T)$ can be formally expanded as a linear combination of $\nabla(L^0)$ and $\nabla(L^1)$. 
Let strand 1 be the strand starting at the NW corner, and let strand 2 be the strand starting at the SW corner. We induct on $N = u_1(T) + o_2(T)$, where $u_1(T)$ is the number of undercrossings of strand 1 not with itself and $o_2(T)$ is the number of overcrossings of strand 2 not with itself. 

Suppose $N=0$. Then strand 1 passes over all other components, and strand 2 passes under all other components. If there is another connected component, it must be split from both strands, and $\nabla(L^T) = 0$. Otherwise, the diagram involving $L^T$ is isotopic to a connected sum, either $L^0\#K_1\#K_2$ or $L^1\#K_1\#K_2$ for some knots $K_1$ and $K_2$, depending on whether strand 1 ends in the NE or SE corner. Specifically, $K_1$ comes from strand 1 and $K_2$ comes from strand 2. Then $\nabla(L^T) = \nabla(K_1)\nabla(K_2)\nabla(L^0)$ or $\nabla(L^T) = \nabla(K_1)\nabla(K_2)\nabla(L^1)$.

Now suppose $N = m$ for $n = k+1$ crossings. Either strand 1 has an undercrossing not with itself, or strand 2 has an overcrossing not with itself. In either case, let $T_0$ be $T$ with this crossing smoothed (with the orientation) and let $T_\times$ be $T$ with this crossing switched. We know that $\nabla(L^{T_0})$ is a linear combination of $\nabla(L^0)$ and $\nabla(L^1)$ by the induction hypothesis on $n$, the number of crossings of $T$. Since $T_\times$ has either $u_1(T_\times) + o_2(T_\times)$ equal to either $m-1$ or $m-2$, $\nabla(T_\times)$ is a linear combination of $\nabla(L^0)$ and $\nabla(L^1)$ by the induction hypothesis on $N$. Then
	\[ \nabla(L^T) = \nabla(L^{T_\times}) \pm z\nabla(L^{T_0}) \]
is also a linear combination of $\nabla(L^0)$ and $\nabla(L^1)$. By induction, the result holds for all left-right oriented tangles $T$.

The proof for diagonally oriented tangles is identical except that strand 1 and strand 2 start at diagonally opposite corners, and $\nabla(L^T)$ is reduced to a linear combination of $\nabla(L^0)$ and $\nabla(L^\infty)$.
\end{proof}


The next lemma motivates the definition of $\con(T)$.
\begin{lem}\label{DiagonalCon}
Let $L^T$ be an oriented link diagram which contains a diagonally oriented tangle diagram $T$.
Then
    \[ \nabla(L^T) 
        = \nabla(T^D)\nabla(L^0) + \nabla(T^N)\nabla(L^\infty). \]
\end{lem}
\begin{proof}
    Note that $0^N$ and $\infty^D$ are split links and thus have Conway polynomial $0$.
    By Theorem \ref{ConReduceThm}, for any oriented link diagram $L^T$ containing the diagonally oriented tangle $T$, the Conway polynomial $\nabla(L^T)$ can be computed as a sum $P(T)\nabla(L^0) + Q(T)\nabla(L^\infty)$, where $P(T),Q(T)\in\Z[z]$ only depend on $T$.
    Recall that $T^N$ and $T^D$ are links containing $T$. If $T$ is diagonally oriented, then there is an induced orientation on $T^N$ and $T^D$. We have 
        \[ \nabla(T^D) = P(T)\nabla(0^D) + Q(T)\nabla(0^N) 
            = P(T)\cdot 1 + Q(T)\cdot 0 = P(T)\]
    and
        \[ \nabla(T^N) = P(T)\nabla(\infty^D) + Q(T)\nabla(\infty^N)
            = P(T)\cdot 0 + Q(T)\cdot 1 = Q(T), \]
as desired.
\end{proof}

\noindent This lemma and the bracket vector definition in \cite{eliahou_infinite_2003} indicate that it is useful to reduce a tangle diagram into a formal sum via a skein relation.

\begin{defn}\label{ConVectorDef}
Let $L^T$ be an oriented link diagram which contains a left-right oriented tangle diagram $T$.
Then $\nabla(L^T)$ can be computed as $p(T)\nabla(L^0) + q(T)\nabla(L^1)$ for some polynomials $p(T),q(T)\in\Z[z]$ by Theorem \ref{ConReduceThm}. Define the Conway vector of $T$ to be
    \[ \con(T) = \begin{bmatrix}p(T)&q(T)\end{bmatrix}^t. \]
\end{defn}

\begin{prop}\label{ConVectorWellDefinedProp}
Let $T$ be a left-right oriented tangle. Then $\con(T)$ is a well-defined invariant of isotopy classes of tangles. Explicitly, $\con(T) = (\nabla((T-1)^N),\nabla(T^N))$.
    \begin{proof}
    We reduce $T^N$ and $(T-1)^N$ using Theorem \ref{ConReduceThm}, since these are both oriented tangle diagrams containing $T$ as a left-right oriented tangle diagram:
        \[ \nabla(T^N) = p(T)\nabla(0^N) + q(T)\nabla(1^N) 
            = p(T)\cdot 0 + q(T)\cdot 1 = q(T),\]
    and
        \[ \nabla((T-1)^N) = p(T)\nabla((0-1)^N) + q(T)\nabla((1-1)^N)
            = p(T)\cdot 1 + q(T)\cdot 0 = p(T). \]
    Then if $T$ is changed by an isotopy, $\nabla(T^N)$ and $\nabla((T-1)^N)$ remain constant. Thus $\con(T)$ is defined for isotopy classes of tangles.
    \end{proof}
\end{prop}

Let $L$ be a diagrammatic operator which takes a single left-right oriented tangle $T$ and yields an oriented link $L(T)$. In particular, $L(T)$ is an oriented link containing the tangle $T$ in some ball. Applying Definition \ref{ConVectorDef}, we get that
    \[ \nabla(L(T)) = p(T)\nabla(L(0)) + q(T)\nabla(L(1)). \]
This proves Theorem \ref{NewInvt}, as this defines a $\Z[z]$-module homomorphism
$\varphi_L : \Z[z]^2 \to \Z[z]$ with
    \[ \varphi_L((p,q)) := p\cdot\nabla(L(0)) + q\cdot\nabla(L(1)) \]
such that $\varphi_L(\con(T)) = \nabla(L(T))$.

The Conway vector gives formulas for various combinations of tangles.
These formulas mirror those from \cite[Prop 2.2]{eliahou_infinite_2003}.
\begin{prop}\label{ConVectorFormulasProp}
Let $T$ be a left-right oriented tangle.
\begin{enumerate}
\item Let $U$ be a left-right oriented tangle. Then
    \[ \con(T+U) = \begin{bmatrix}p(T)p(U) + q(T)q(U)\\
		p(T)q(U) + q(T)p(U) + zq(T)q(U)\end{bmatrix}. \]
\item Let $W$ be a diagonally oriented tangle such that $T*W$ is a left-right oriented tangle. Then
    \[ \con(T*W) = \begin{bmatrix}\nabla(W^N)&\nabla(W^D)\\
        0 & \nabla(W^N)\end{bmatrix}\con(T). \]
\end{enumerate}
    \begin{proof}  
    First, we use Definition \ref{ConwayPolynomialDef} to see that, if $L^2$ is any link diagram containing the integer tangle $2$, then
        \[ \nabla(L^2) = \nabla(L^0) + z\nabla(L^1). \]
    Note that if $T$ and $U$ are left-right oriented tangles, then so is $T+U$.  Consider an oriented  link diagram $L^{T+U}$ containing $T+U$ in some disk intersecting the link diagram in four points. Reducing $T$ and $U$ in turn using the skein relation,
        \begin{align*}
        \nabla(L^{T+U}) &= p(T)p(U)\nabla(L^0) 
            + (p(T)q(U) + q(T)p(U))\nabla(L^1) + q(T)q(U)\nabla(L^2) \\
            & = (p(T)p(U) + q(T)q(U))\nabla(L^0) 
                + (p(T)q(U) + q(T)p(U)+zq(T)q(U))\nabla(L^1).
        \end{align*}
    This proves the first statement.
    
    To prove the second statement,
    we first note that $((T*0)-1)^N$ is isotopic to $T^N$, $T*\infty = T$, and $(T*0)^N = T^N\sqcup O$ is a split link which thus has Conway polynomial $0$.  Recall that the Conway polynomial is a link invariant and therefore does not change under isotopy.  We can treat diagrams of $((T*W)-1)^N$ and $(T*W)^N$ as oriented link diagrams containing a diagonally oriented tangle $W$.
    Applying Lemma \ref{DiagonalCon}, Definition \ref{ConVectorDef}, and Proposition \ref{ConVectorWellDefinedProp}, we see that
        \begin{align*}
        p(T*W) &= \nabla(((T*W)-1)^N)\\
        & = \nabla(W^D)\nabla(((T*0)-1)^N) + \nabla(W^N)\nabla(((T*\infty)-1)^N) \\
        & = \nabla(W^D)\nabla(T^N) + \nabla(W^N)\nabla((T-1)^N) \\
        & = \nabla(W^N)p(T) + \nabla(W^D)q(T), \\
        q(T*W) &= \nabla((T*W)^N) \\
        & = \nabla(W^D)\nabla((T*0)^N) + \nabla(W^N)\nabla((T*\infty)^N)\\
        & = 0 + \nabla(W^N)\nabla(T^N) = \nabla(W^N)q(T). \qedhere
        \end{align*}
    \end{proof}
\end{prop}

\section{Conway Polynomials of 2-sky-like Links}

In this section, we find an infinite family of 2-sky-like links with Conway polynomial indistinguishable from $H$, the connected sum of two Hopf links. We show, using the Jones polynomial, that all of the links in this family are distinct from each other and distinct from $H$. More specifically, we will show that
    \begin{equation}\label{ConEqual} \nabla(C_+(T)) = \nabla(C_-(T)) = z^2 
        = \nabla(C_+(0)) = \nabla(C_-(0)) \end{equation}
for infinitely many left-right oriented tangles $T$. First, we identify a necessary and sufficient condition for this equation to be satisfied.
\begin{lem}\label{ConEqualLem}
Equation \ref{ConEqual} is satisfied by a tangle $T$ if and only if $\con(T) = (1,0)$.
    \begin{proof}
    We note that $C_+(T)$ and $C_-(T)$ are oriented links containing a left-right oriented tangle $T$.
    Using the Definition \ref{ConVectorDef} of $\con(T)$ and the fact that $C_+(1)$ is L4a1\{1\} and $C_-(1)$ is L4a1\{0\} in \cite{linkinfo}, we have
        \begin{align*}
        \nabla(C_+(T)) &= p(T)\nabla(C_+(0)) + q(T)\nabla(C_+(1)) \\
        & = (z^2)p(T) + (2z+z^3)q(T),
        \end{align*}
    and
        \[ \nabla(C_-(T)) = (z^2)p(T) - (2z)q(T). \]
    If $\con(T) = (1,0)$, then Equation \ref{ConEqual} is satisfied. Conversely, suppose Equation \ref{ConEqual} is satisfied.
    Since $\nabla(C_+(T)) = \nabla(C_-(T))$, we must have $q(T) = 0$, and it follows that $p(T) = 1$.
    \end{proof}
\end{lem}

\begin{figure}
    \centering
    \includegraphics[width=2in]{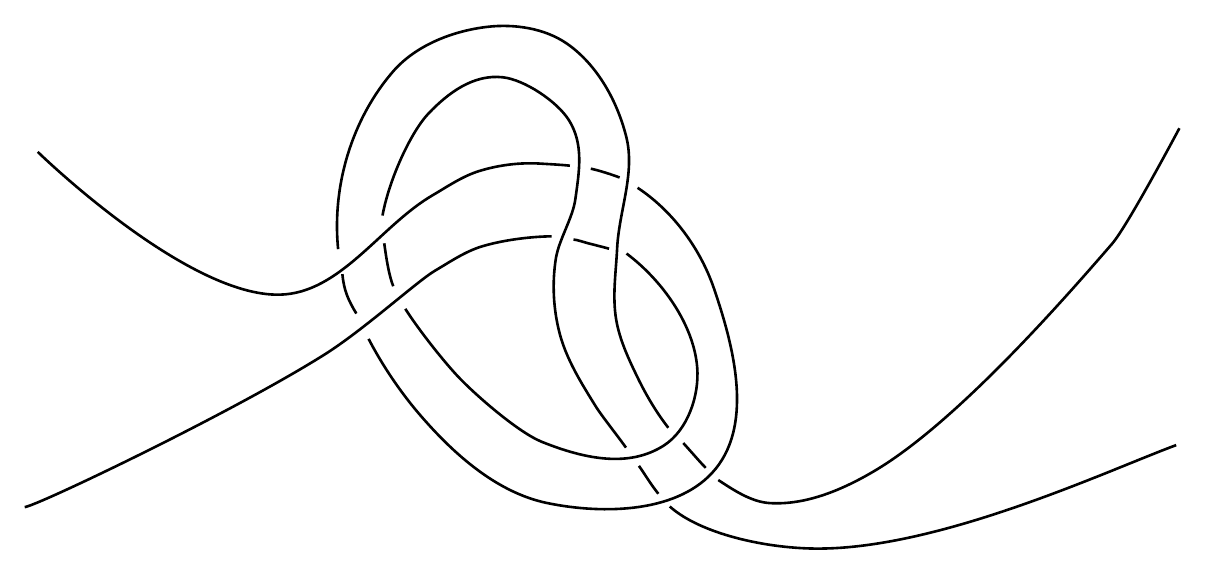}
    \caption{The tangle $T_C$.}
    \label{fig:T_C}
\end{figure}

We now construct solutions $T$ to Equation \ref{ConEqual}. In \cite[107]{kauffman_conway_1981}, Kauffman identifies three diagonally oriented tangles $T_A = 2^\rho$, $T_B = -6$, and $T_C$ with fractions
$F(T_A) = 1/z$, $F(T_B) = 3z/1$, and $F(T_C) = -3z/1$. We refer the reader to Section \ref{sec:background} for definitions of $T^\rho$ and integer tangles.  For a diagram of $T_C$, see Figure \ref{fig:T_C}.
Using Theorem \ref{TangleFractionSumThm},
    \[ F(T_A + T_B + T_C) = 1/z + 3z/1 - 3z/1 = 1/z + 0/(1\cdot 1)
        = (1\cdot 1 + 0\cdot z)/(z\cdot 1) = 1/z. \]
One can check that with the same orientation, $F(-T_A) = 1/(-z)$. Then a similar calculation gives
    \[ F(-T_A + T_B + T_C) = 1/(-z). \]
Let $T_+$ be $T_A + T_B + T_C$ rotated 90 degrees clockwise, and $T_-$ be $-T_A + T_B + T_C$ rotated 90 degrees clockwise. Since the rotation swaps numerator and denominator closures, $F(T_+) = z/1$ and $F(T_-) = -z/1$.
Let $T_0 = T_+ + T_-$.  Then $T_0$ is a diagonally oriented tangle which has
    \[ F(T_0) = F(T_+) + F(T_-) = (z\cdot 1 + (-z)\cdot 1)/(1\cdot 1) = 0/1. \]
In other words, $\nabla(T_0^N) = 0$ and $\nabla(T_0^D) = 1$.
Let $T_0(n)$ denote an $n$-fold sum $T_0 + \cdots + T_0$.
Then $F(T_0(n)) = 0/1$ for all positive integers $n$ by repeated application of Theorem \ref{TangleFractionSumThm}. If the tangle $1$ is given a left-right orientation, then $1*T_0(n)$ is left-right oriented.
The link $C(1*T_0(1))$ is shown in Figure \ref{fig:JSlink}. The links $C_+(1*T_0(n))$ and $C_-(1*T_0(n))$ are visibly 2-sky-like for all $n \in \Z_+$.

\begin{figure}
    \centering
    \includegraphics[width=3in]{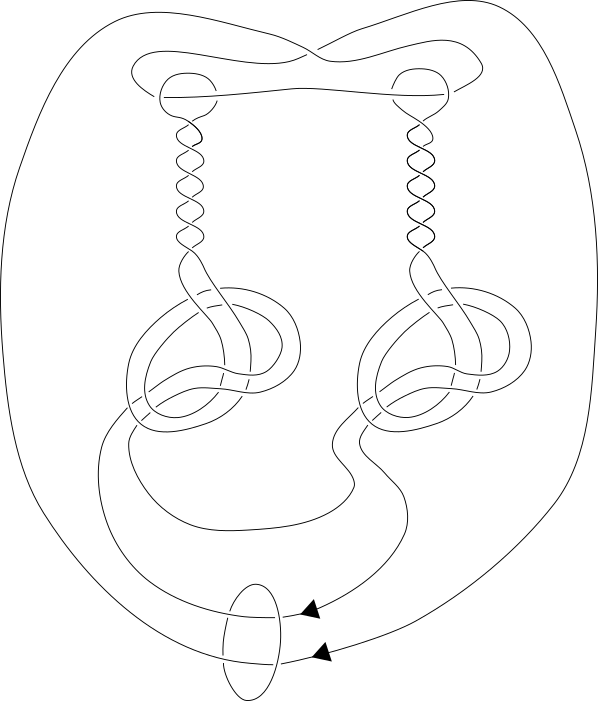}
    \caption{$C(1*T_0(1))$.}
    \label{fig:JSlink}
\end{figure}

\begin{thm}\label{JSLinkThm}
For any positive integer $n$, Equation \ref{ConEqual} is satisfied by $1*T_0(n)$, and the (unoriented) link $C(1*T_0(n))$ is not isotopic to the (unoriented) connected sum of two Hopf links $H$. The links $C(1*T_0(n))$ are distinct as unoriented links. Thus, there are infinitely many 2-sky-like links indistinguishable from $H$ by the Conway polynomial.
\begin{proof}
We show that the span of the bracket polynomial distinguishes all of the links $C(1*T_0(n))$, $n \in \Z_+$, from each other and from $C(0)$. With the help of \verb+KnotTheory`+ in Mathematica \cite{Mathematica}, we compute $\langle T_0^N\rangle$ and $\langle T_0^D\rangle$. By Definition \ref{BracketVectorDef}, for any tangle $T$,
    \[ \langle T^N\rangle = [\langle 0^N\rangle\quad\langle \infty^N\rangle]\br(T) 
        = [\delta\quad 1]\br(T), \]
where $\delta = -A^2-A^{-2}$. Similarly,
    \[ \langle T^D\rangle = [1\quad\delta]\br(T). \]
Thus,
    \[ \begin{bmatrix}\langle T^N\rangle\\
        \langle T^D\rangle\end{bmatrix}
        = \begin{bmatrix}\delta&1\\1&\delta\end{bmatrix}\br(T), \]
and so
    \[ \br(T) = \begin{bmatrix}f(T)\\g(T)\end{bmatrix}
        = \begin{bmatrix}\delta/(\delta^2-1)&-1/(\delta^2-1)\\
        -1/(\delta^2-1)&\delta/(\delta^2-1)\end{bmatrix}
        \begin{bmatrix}\langle T^N\rangle\\
        \langle T^D\rangle\end{bmatrix}. \]
For ease of notation, let $\mathrm{Poly}(y_1;y_2)\subset\Z[A,A^{-1}]$ be the set of Laurent polynomials with highest degree monomial $y_1$ and lowest degree monomial $y_2$.
Using the above formula, we calculate that
    \[ f(T_0) \in \mathrm{Poly}(A^{40};-A^{-52}),\quad
        g(T_0) \in \mathrm{Poly}(-A^{34};A^{-58}).  \]
By Proposition \ref{BracketVectorSumProp}, for $n \in \Z_+$,
    \[ \br(T_0(n+1)) =\br(T_0(n)+T_0)= \begin{bmatrix}f(T_0(n))f(T_0)\\
        f(T_0(n))g(T_0)+g(T_0(n))f(T_0)+\delta g(T_0(n))g(T_0)\end{bmatrix}.\]
Assume for some $n$ that
    \begin{equation}\label{br(T_0(n))Eq} f(T_0(n)) \in \mathrm{Poly}(A^{40n};(-1)^nA^{-52n})\quad\text{and}\quad g(T_0(n)) \in \mathrm{Poly}(-nA^{40n-6};(-1)^{n+1}A^{-60n+2}), \end{equation}
which is true for $n=1$. Then
    \[ f(T_0(n+1)) = f(T_0(n))f(T_0) 
        \in \mathrm{Poly}(A^{40(n+1)};(-1)^{n+1}A^{-52(n+1)}). \]
Similarly,
    \[  f(T_0(n))g(T_0) \in \mathrm{Poly}(-A^{40(n+1)-6};(-1)^nA^{-52n-58}), \]
    \[ g(T_0(n))f(T_0) \in \mathrm{Poly}(-nA^{40(n+1)-6};(-1)^{n+2}A^{-60(n+1)+10}), \]
    \[ \delta g(T_0(n))g(T_0) 
        \in \mathrm{Poly}(nA^{40(n+1)-10};(-1)^{n+2}A^{-60(n+1)+2}), \]
and so
    \[ g(T_0(n+1))
        \in \mathrm{Poly}(-(n+1)A^{40(n+1)-6};(-1)^{n+2}A^{-60(n+1)+2}). \]
By induction, equation \ref{br(T_0(n))Eq} holds for all $n\in\Z_+$.
From Definition \ref{BracketVectorDef},
    \[ \langle C(1*T_0(n))\rangle = f(T_0(n))\langle C(1*0)\rangle 
        + g(T_0(n))\langle C(1*\infty)\rangle. \]
We calculate that $\langle C(1*0)\rangle = -A^{11}-2A^{3}-A^{-5}$ and $\langle C(1*\infty)\rangle = A^9+A-A^{-3}+A^{-7}$. So for $n \in \Z_+$,
    \[ f(T_0(n))\langle C(1*0)\rangle \in
        \mathrm{Poly}(-A^{40n+11};(-1)^{n+1}A^{-52n-5}), \]
    \[ g(T_0(n))\langle C(1*\infty)\rangle \in
        \mathrm{Poly}(-nA^{40n+3};(-1)^{n+1}A^{-60n-5}), \]
and so
    \[ \langle C(1*T_0(n))\rangle \in
        \mathrm{Poly}(-A^{40n+11};(-1)^{n+1}A^{-60n-5}). \]
Thus,
    \[ \mathrm{span}(V(C(1*T_0(n)))) 
        = \mathrm{span}(\langle C(1*T_0(n))\rangle) = 100n+16 > 16. \]
Then the Jones polynomial of $C(1*T_0(n))$ is distinct for distinct $n \in \Z_+$ and not equal to the Jones polynomial of $C(0)$ for any orientation of the components. It follows that the links $C(1*T_0(n))$ are distinct from each other and distinct from $C(0)$.

It remains to show that $1*T_0(n)$ satisfies Equation \ref{ConEqual}. By Lemma \ref{ConEqualLem}, it is sufficient to show that $\con(1*T_0(n)) = (1,0)$. 
Recall that $\nabla(T_0(n)^N) = 0$ and $\nabla(T_0(n)^D) = 1$, and observe that $\con(1) = [0\quad 1]^t$.
Using Proposition \ref{ConVectorFormulasProp},
    \[ \con(1*T_0(n)) = \begin{bmatrix}0&1\\0&0\end{bmatrix}
    \begin{bmatrix}0\\1\end{bmatrix} = \begin{bmatrix}1\\0\end{bmatrix}.\]
This gives the desired result.
\end{proof}
\end{thm}

\begin{rmk}
We have not proved conclusively that the Conway polynomial does not detect causality in the given setting. The link $C(1*T_0(n))$ might not correspond to the skies of a pair of causally related events. However, $C(1*T_0(n))$ satisfies the basic topological restrictions of a 2-sky-like link, giving strong evidence that the Conway polynomial does not detect causality in the given setting.
\end{rmk}

\begin{rmk}
Though it is not directly relevant to this setting, the Conway vector cannot distinguish the tangle $T_0(n)$ from the tangle $0$ even when the orientation of one of the strands is reversed. Let $T_A' = (-2)^\rho$, left-right oriented, and let $U_0$ be $T_B + T_C$ rotated 90 degrees clockwise and oriented as above.
Note that $T_0':=(T_A'*U_0)+((-T_A')*U_0)$ is $T_0$ but left-right oriented.
The above calculation gives that $\nabla(U_0^N) = 1$ and $\nabla(U_0)=0$, so
Proposition \ref{ConVectorFormulasProp} gives $\con(T_A'*U_0)=\con(T_A')$ and $\con((-T_A')*U_0)=\con(-T_A')$. Then $\con(T_0') = \con(T_A'+(-T_A')) = \con(0)$, since Proposition \ref{ConVectorFormulasProp} gives the Conway vector of a sum of tangles in terms of the Conway vectors of the tangle summands.
Thus, $T_0$ is indistinguishable from the tangle $0$ via the Conway polynomial no matter how it is oriented.
A similar calculation to the above will show that $C_+(1*T_0(n))$ and $C_-(1*T_0(n))$, when $T_0$ is left-right oriented, has the same Conway polynomial as $C_+(0)$ and $C_-(0)$ with the longitudinal components oriented oppositely relative to the meridional component. Thus the Conway polynomial cannot distinguish $C(1*T_0(n))$ from the connected sum of two Hopf links $H$.
\end{rmk}

\printbibliography

\end{document}